\documentclass[12pt]{article}
\usepackage{amssymb}
\usepackage{mathrsfs}
\usepackage{amsfonts}
\usepackage{graphicx}
\usepackage{indentfirst}
\usepackage{colortbl,dcolumn}
\usepackage{color}
\usepackage{comment}
\usepackage{xcolor}
\usepackage{amsmath}
\usepackage{bbm} 
\usepackage{dsfont} 
\usepackage{enumerate}
 \usepackage{comment}
\usepackage{psfrag}
\usepackage{booktabs}
\usepackage{array}
\usepackage{cite}
\usepackage{amsthm}
\numberwithin{equation}{section}
\usepackage[top=1in, bottom=1in, left=0.9in, right=0.9in]{geometry}
\usepackage{authblk}
\usepackage{caption}
\usepackage{hyperref}

\newcommand{\R}{\mathbb{R}}
\newcommand{\N}{\mathbb{N}}
\newcommand{\E}{\mathbb{E}}
\renewcommand{\P}{\mathscr{P}}

\newcommand{\F}{\mathcal{F}}

\newtheorem{thm}{Theorem}[section]

\newtheorem{lem}[thm]{Lemma}

\newtheorem{rem}[thm]{Remark}
\newtheorem{example}[thm]{Example}
\newtheorem{assumption}[thm]{Assumption}

\hypersetup{colorlinks=true,            
	linkcolor=red,
	anchorcolor=red,
	citecolor=red}

\begin{document}
\title{
Explicit modified Euler 
approximations of the A\"{i}t-Sahalia type model with Poisson jumps
\thanks{
This work was supported by the Natural Science Foundation of China (12471394, 12071488, 12371417),
the Fundamental Research Funds for the Central Universities of Central South University (2024ZZTS0307).
}
}
\author[a]{Yingsong Jiang\protect\footnotemark[4]}
\author[b,c]{Ruishu Liu\thanks{Corresponding author:
rsliu@szu.edu.cn
}}
\author[a]{Minhong Xu\protect\footnotemark[4]}
\affil[a]{\footnotesize School of Mathematics and Statistics, Hunan Research Center of the Basic Discipline for Analytical  
\\
\footnotesize Mathematics, HNP-LAMA, Central South University, Changsha 410083, Hunan, China}
\affil[b]{School of Artificial Intelligence, Shenzhen University, China,}
\affil[c]{National Engineering Laboratory for Big Data System Computing Technology,
Shenzhen University, China,}

\date{ }
\maketitle

\footnotetext[4]{Contributing author(s):
\begin{minipage}[t]{\linewidth}
    yingsong@csu.edu.cn,
    13476813022@163.com.
\end{minipage}
}

\begin{abstract}
This paper focuses on mean-square approximations of a generalized A\"it-Sahalia interest rate model with Poisson jumps. The main challenge in the construction and analysis of time-discrete numerical schemes is caused by a drift that blows up at the origin, highly nonlinear drift and diffusion coefficients and positivity-preserving requirement. Due to the presence of the Poisson jumps, additional difficulties arise in recovering the exact order $1/2$  of convergence for the time-stepping schemes. By incorporating implicitness in the term $\alpha_{-1}x^{-1} $ and introducing
    the modifications functions $f_h$ and $g_h$ in the recursion, a novel explicit Euler-type scheme is proposed, which is easy to implement and preserves the positivity of the original model unconditionally, i.e., for any time step-size $h>0$. 
    A mean-square convergence rate of order $1/2$ is established for the proposed scheme in both the non-critical and general critical cases. Finally, numerical experiments are provided to confirm the theoretical findings. 

    \par\textbf{keywords: } A\"{i}t-Sahalia model with Poisson jumps; Unconditionally positivity preserving; Explicit Euler-type scheme; Mean-square convergence rate
\end{abstract}

\section{Introduction}

Stochastic differential equations (SDEs) are extensively utilized across various scientific disciplines to model real-world phenomena influenced by random noise.
For example, in the field of mathematical finance, SDEs play a vital role in capturing the dynamics of financial systems. Prominent examples include the Black-Scholes model, the Cox-Ingersoll-Ross (CIR) model, and the A\"it-Sahalia model, which are foundational to modern financial theory.
Moreover, to accurately model event-driven phenomena, it is essential to incorporate SDEs with Poisson jumps  \cite{cont2003financial, platen2010numerical}. For instance, stock price movements can be significantly impacted by sudden, unpredictable events such as market crashes, central bank announcements, or changes in credit ratings. Over the past decade, SDEs with jumps have gained increasing popularity for modeling market fluctuations, particularly in the contexts of risk management and option pricing \cite{cont2003financial}. Given that analytical solutions for nonlinear SDEs with jumps are seldom available, numerical approximations have become an indispensable tool for understanding the behavior of these complex systems. This has motivated a substantial body of research dedicated to exploring this intriguing and challenging topic (see, e.g., \cite{yang2024strong,chen2020convergence,Chen_Gan_Wang:SIS_JCAM21,dareiotis2016tamed,deng2019truncated,higham2005numerical,higham2006convergence,higham2007strong,kaluza2018optimal,kumar2016tamed,li2019positivity,platen2010numerical,przybylowicz2016optimal,wang2010compensated,yang2017transformed}).

Throughout this paper we study
the generalized A\"it-Sahalia interest rate model with Poisson jumps, given by:
\begin{equation}\label{eq:ait-sahalia_general[0.5AS25]}
      \mathrm{d} X_t 
     =
     ( \alpha_{-1} X_t^{-1} - \alpha_0 + \alpha_1 X_t - \alpha_2 X_t^r ) \, \mathrm{d}t 
     +
     \sigma X_t^{\rho} \, \mathrm{d} W_t
     +
     \nu( X_{t^-} ) \, \mathrm{d} N_t
     ,
     \ t >0, \quad  X_0 = x_0 > 0,  
\end{equation}
where $ \alpha_{-1}, \alpha_0, \alpha_1, \alpha_2 >0$, $r, \rho >1$
such that $r+1 \geq 2\rho$, 
$\nu : \R \rightarrow \R$, and $X_{t^-} := \lim_{r \rightarrow t^-} X_r$. 
Here, $\{W_t\}_{t \in [0, \infty)}$ denotes a standard Brownian motion and $\{N_t\}_{t \in [0, \infty)}$ a Poisson process with jump intensity $\lambda >0$,
both defined on a filtered probability space $(\Omega, \mathcal{F}, \mathbb{P}, \{\mathcal{F}_t\}_{t \geq 0})$ with respect to the normal filtration $\{\mathcal{F}_t\}_{t \geq 0}$. 
We further assume that
the Brownian motion and the Poisson process are independent. 
The compensated Poisson process is defined as $\tilde{N}_t = N_t - \lambda t$ for $t \in [0, \infty)$, which is a martingale with respect to the normal filtration $\{\mathcal{F}_t\}_{t \geq 0}$. 
When $\nu \equiv 0$, the model reduces to the generalized A\"it-Sahalia interest rate model (without jumps), which has been intensively studied (see, e.g., \cite{szpruch2011numerical,higham2002strong,neuenkirch2014first,wang2020mean,liu2025unconditionally,jiang2024unconditionally}).
Under suitable parameter conditions,
the model (\ref{eq:ait-sahalia_general[0.5AS25]}) admits a unique positive solution.
Preserving such positivity is therefore critical for designing numerical schemes.
Indeed, numerous nonlinear SDE models arising in finance possess not only non-globally Lipschitz coefficients but also positive constraints, which has attracted significant interest in developing positivity-preserving numerical methods.
The standard Euler–Maruyama (EM) method, despite its ease of implementation and widespread use \cite{kloeden1992numerical,HK_book}, often fails in such settings.
It diverges in the strong sense owing to the super‑linear coefficients \cite{hutzenthaler2011strong},
and can not preserve positivity.
To address these limitations, 
the implicit methods \cite{higham2002strong,wang2020mean,wang2023mean,zhao2021backward} and modified explicit methods \cite{hutzenthaler2012strong,mao2015truncated,hutzenthaler2020perturbation,cai2022positivity,yang2024strong,yi2021positivity,liu2023higher,jentzen2009pathwise,li2019explicit}
have been developed.
While implicit methods naturally preserve positivity,
they require solving an implicit algebraic equation in each step, which increases computational cost. 
To improve efficiency, 
various modified explicit methods have been proposed, often relying on taming, projection, or truncation techniques.
Among these, tamed schemes effectively control super-linear growth but do not guarantee positivity. In contrast, projection and truncation methods preserve positivity by modifying the numerical solution; however, they may introduce bias and complicate the convergence analysis.

Despite significant progress, 
numerical methods for the A\"it-Sahalia interest rate model with Poisson jumps are less explored.
Compared to the generalized A\"it-Sahalia interest rate model (without jumps), 
the Poisson process brings
additional difficulties that arise in recovering the exact convergence order for the time-stepping schemes. 
In the presence of a linear jump function $\nu$,
Deng et al. \cite{deng2019generalized} studied the analytical properties of the model (\ref{eq:ait-sahalia_general[0.5AS25]}) and applied
the Euler-Maruyama method to this model, proving convergence in probability but without quantifying the convergence rate. 
Recently, Zhao et al. \cite{zhao2021backward} employed the backward Euler method (BEM) to (\ref{eq:ait-sahalia_general[0.5AS25]}),
obtaining a mean-square strong convergence rate of order $1/2$ under the condition $\gamma + 1 \geq 2\rho$,
with some restrictions imposed on the time step size $h$.

In this paper, we aim to propose new explicit schemes 
for approximating the solution to \eqref{eq:ait-sahalia_general[0.5AS25]}, which is strongly convergent and positivity preserving in both the non-critical and critical cases.
Given the interval $[0, T]$, we construct a uniform mesh with time step size $h = \frac{T}{N}$, where $N$ is a positive integer. 
For $ n \in \{0,1,2,...,N-1\}$, the proposed time stepping scheme is defined by
\begin{equation}\label{eq:numer._in_intro.[0.5AS25]}
        Y_{n+1}
       =
       Y_n
       +  (\alpha_{-1} Y_{n+1}^{-1} - \alpha_0 + \alpha_1 Y_n + f_h (Y_n)    ) h
       + g_h (Y_n)  \Delta W_n 
       + \nu( Y_n ) \Delta N_n,
\end{equation}
with $Y_0 = x_0$. 
Here, the terms $\Delta W_n = W_{t_{n+1}} - W_{t_n}$ and $\Delta N_n = N_{t_{n+1}} - N_{t_n}$ represent the increments of the Wiener process and Poisson process, respectively.
Specifically, $f_h(x)$ and $g_h(x)$ are step-size-dependent modifications of $f(x) := -\alpha_2 x^r$ and $g(x) := \sigma x^\rho$ satisfying Assumption~\ref{assump:modification_f-g[0.5AS25]}.
These modifications are designed to address the challenges posed by the polynomial growth of the coefficients.
Two concrete examples are provided in Example~\ref{example1} and Example~\ref{example2[0.5AS25]}.
Notably, the proposed Euler-type scheme is actually explicit, and unconditionally preserves positivity for any time step size $h > 0$ (see Lemma \ref{lem:wellposed_of_Numerical method[0.5AS25]}).
Inspired by \cite{liu2025unconditionally},
the implicitness is solely present in the term $\alpha_{-1} Y_{n+1}^{-1} h$, which plays a crucial role in preserving positivity.
Given the previous step $Y_n$, the current step $Y_{n+1}$
is computed as the unique positive root
of a quadratic equation (see \eqref{eq:solution_of_scheme} for the explicit solution).
This approach significantly reduces computational costs compared to implicit schemes like BEM reported in \cite{zhao2021backward},
and allows for efficient implementation.

To the best of our knowledge, this is the first work to develop an explicit scheme that unconditionally preserves positivity and achieves a mean-square convergence rate of $1/2$ for the general A\"{i}t-Sahalia model with Poisson jumps \eqref{eq:ait-sahalia_general[0.5AS25]} in both non-critical and general critical cases.
We also mention that,
by developing a novel framework of correction functions $f_h$ and $g_h$ that address polynomially growing coefficients, 
the present work thereby simplifies the strong convergence analysis and avoids the requirement for any a priori moment bounds of the numerical schemes.

The rest of this paper is organized as follows. The next section provides some preliminaries for the considered model. In Section \ref{section:scheme[0.5AS25]}, we present the numerical method and its unconditionally positivity preserving property, along with two concrete schemes presented to meet the conditions of the proposed method.
The mean-square convergence analysis is conducted in Section \ref{section:error analysis[0.5AS25]}, where the convergence rate is established by Theorem \ref{thm(revised_last):Order_onehalf[0.5AS25]}. Numerical experiments are carried out in Section \ref{section:numericalexperiments[0.5AS25]} to validate the theoretical results.

\section{Preliminaries}\label{section:preliminaries[0.5AS25]}
Let $\N$ be the set of all positive integers.
By $| \cdot |$ we denote the Euclidean norm in $\R$.
Given $T \in (0, \infty)$ and
a filtered probability space $(\Omega, \F, \{ \F_t \}_{t\in [0,T]}, \mathbb{P} )$, we use $\E$ to represent the expectation and $L^p(\Omega;\R), p > 0$ to represent the space of all $\R$-valued random variables $\eta$ satisfing $\E[| \eta |^p] < \infty $, equipped with the norm $\| \cdot \|_{ L^p(\Omega;\R) }$ defined by:
\begin{equation}
    \| \eta \|_{ L^p(\Omega;\R) } := ( \E[ | \eta |^p ] )^{ \frac{1}{p} }, \quad \text{for all }\eta \in L^p(\Omega;\R), \ 
    p > 0.
\end{equation}

This paper considers
the A\"{i}t-Sahalia type model with Poisson jumps of the form
\begin{equation}\label{eq:ait-sahalia[0.5AS25]}
      \mathrm{d} X_t = ( \alpha_{-1} X_t^{-1} - \alpha_0 + \alpha_1 X_t + f( X_t ) ) \, \mathrm{d}t + g( X_t ) \, \mathrm{d} W_t + \nu(X_{t^-}) \, \mathrm{d}N_t, \ t > 0, \quad  X_0 = x_0 > 0,
\end{equation}
where for short we denote 
\begin{equation}\label{eq:f_and_g[0.5AS25]}
    f(x) := - \alpha_2 x^r , 
    \ \
    g(x) := \sigma x^{\rho}, 
    \quad
     x \in \R,
\end{equation}
with the parameters $\alpha_{-1}, \alpha_0, \alpha_1, \alpha_2, \sigma > 0$ and $r,\rho > 1$.
We additionally assume that the
coefficient function \(\nu \colon \mathbb{R} \rightarrow \mathbb{R}\) in (\ref{eq:ait-sahalia[0.5AS25]}) satisfies
\begin{equation}\label{eq:nu_assmption1[0.5AS25]}
|\nu(x) - \nu(y)| \leq M|x - y|, \quad \text{for all } x, y \in D, 
\end{equation}
and
\begin{equation}\label{eq:nu_assmption2[0.5AS25]}
x + \nu(x) > m \cdot \min\{1, x\}, \quad \text{for all } x \in D,
\end{equation}
for some constants \(M , m > 0\).
It is also evident from inequality \eqref{eq:nu_assmption1[0.5AS25]} that
\begin{equation}\label{eq:nu_assmption3[0.5AS25]}
|\nu(x)| \leq C(1 + |x|), \quad \text{for all } x \in D.
\end{equation}
Here, and throughout, the letter \(C\) denotes a generic positive constant, which is independent of the step-size and may vary at different appearances.
Besides, a key component in our analysis is the compensated Poisson process, defined as 
\begin{equation}
    \tilde{N}_t := N_t - \lambda t , \quad
    t \in [0, \infty),
\end{equation}
which is a martingale with respect to the normal filtration $\{\mathcal{F}_t\}_{t \geq 0}$.

Under the above settings,
the well-posedness of the A\"{i}t-Sahalia type model with Poisson jumps \eqref{eq:ait-sahalia[0.5AS25]} is established in the following lemma, quoted directly from
\cite[Proposition 1]{zhao2021backward}.
\begin{lem}\label{wellposed of exact solution[0.5AS25]}  
Suppose the conditions \eqref{eq:nu_assmption1[0.5AS25]}-\eqref{eq:nu_assmption2[0.5AS25]} hold. 
For any given initial data $X_0 = x_0 >0$, 
the model \eqref{eq:ait-sahalia[0.5AS25]} admits a unique positive $\{\mathcal{F}_t\}_{t \geq 0}$-adapted global solution $\{X(t)\}_{t \in [0,T]}$.
\end{lem}

The next lemma summarizes the results of the moment bounds for the solution of model \eqref{eq:ait-sahalia[0.5AS25]} in both the non-critical case $r +1 > 2\rho$ and the critical case $r+1=2\rho$
(cf. \cite[Lemma 1, Lemma 2]{zhao2021backward}).
\begin{lem}\label{lem:moment_case1[0.5AS25]}
    Let the conditions \eqref{eq:nu_assmption1[0.5AS25]}-\eqref{eq:nu_assmption2[0.5AS25]} hold and
    let $\{ X_t \}_{ t\in [0,T] }$ be the solution of (\ref{eq:ait-sahalia[0.5AS25]}). 

\begin{enumerate}
    \item if $r + 1 > 2 \rho$, then for any \(p_0 \geq 2\) and $p \geq  1$, it holds that
\begin{equation}
    \sup_{t \in [0,T]} \mathbb{E}[|X_t|^{p_0}] < \infty, \quad \sup_{t \in [0,T]} \mathbb{E}[|X_t|^{-p}] < \infty. 
\end{equation}
 
    \item if $r + 1 = 2 \rho$, then for any \(2 \leq p_1 \leq \frac{\sigma^2 + 2\alpha_2}{\sigma^2}\) and $p \geq 1$, it holds that
\begin{equation}
\sup_{t \in [0,T]} \mathbb{E}[|X_t|^{p_1}] < \infty, \quad \sup_{t \in [0,T]} \mathbb{E}[|X_t|^{-p}] < \infty.
\end{equation}
\end{enumerate}
\end{lem}

Equipped with Lemma \ref{lem:moment_case1[0.5AS25]}, 
the H\"older continuous properties of the solution $\{ X_t \}_{ t\in [0,T] }$ to model \eqref{eq:ait-sahalia[0.5AS25]} can also be derived, as given in \cite{zhao2021backward} (see Lemmas 4, 5 and the equations (59), (78) therein).
Here, we collect the results in the following lemma.
\begin{lem}
\label{lem:Holder_continuous_X[0.5AS25]}
    Under same assumptions as in Lemma \ref{wellposed of exact solution[0.5AS25]},
    for either the non-critical case $r + 1 > 2 \rho$, or the critical case $r + 1 = 2 \rho$ with $\tfrac{\alpha_2}{\sigma^2} > 2r - \tfrac32$,
    it holds for all $t,s \in[0,T]$ that
    \begin{align}
       \Vert X_t - X_s \Vert_{L^{2}(\Omega; \mathbb{R})}
         \leq 
        C \vert t-s \vert^{\frac12},
        &
        \quad
        \Vert X_t^{-1} - X_s^{-1} \Vert_{L^{2}(\Omega; \mathbb{R})}
         \leq 
        C \vert t-s \vert^{\frac12}.
    \end{align}
    
\end{lem}

Before closing this section, we present the monotonicity condition for
$f$ and $g$ in the following lemma, 
whose proof can be found in \cite[Lemma 5.9, Lemma 5.12]{wang2023mean}.

\begin{lem}     
\label{lemma:f_g_monotonicity[0.5AS25]}
    Let $f$ and $g$ be defined by \eqref{eq:f_and_g[0.5AS25]}.
    If the parameters in the model \eqref{eq:ait-sahalia[0.5AS25]} satisfy
    \begin{equation} \label{eq:parameter_conditions[0.5AS25]}
\begin{aligned}
    &\text{(a)} \, r + 1 > 2\rho \quad
    \text{or}
    \quad
    \text{(b)} \, r + 1 = 2\rho, \quad \tfrac{\alpha_2}{\sigma^2}
    >
    \tfrac{1}{8} \left(r + 2 + \tfrac{1}{r} \right),
\end{aligned}
\end{equation}
    then there exist constants $q > 2$ and $L > 0$ such that for all $x \in D$, 
    \begin{equation} \label{eq:monotonicity[0.5AS25]}
     f'(x) + \tfrac{q- 1}{2} | g'(x) |^2 
     \leq L .
    \end{equation}
\end{lem}

We mention that the above lemma is not directly used in the convergence analysis. Instead, it plays a crucial role in the design of numerical schemes. In particular, \eqref{eq:parameter_conditions[0.5AS25]} enables the verification of certain monotonicity condition (see \eqref{eq:assump.(x-y)(f_h(x) - f_h(y))[0.5AS25]}) when taming-based or projection-based correction functions are used, 
as will be discussed in the next section. 
Therefore, throughout this paper, we assume that \eqref{eq:parameter_conditions[0.5AS25]} holds.

\section{A class of explicit Euler-type schemes}\label{section:scheme[0.5AS25]}

In this section we propose a class of explicit Euler-type schemes for the A\"{i}t-Sahalia type model with Poisson jumps \eqref{eq:ait-sahalia[0.5AS25]}, with two concrete examples of the schemes provided.
Throughout, we perform the temperal discretization on a uniform mesh within the interval $[0,T]$, by setting
a uniform step-size $h = \tfrac{T}{N}, N \in \mathbb{N} $ and denoting the grid points $t_n := nh, n \in \{ 0,1, ..., N  \}$.

Then we propose
the explicit Euler-type scheme for the model \eqref{eq:ait-sahalia[0.5AS25]}
 as follows:
\begin{equation}\label{eq:numeri_method[0.5AS25]}
 Y_{n+1}
       =
      Y_n
       +  (\alpha_{-1} Y_{n+1}^{-1} - \alpha_0 + \alpha_1 Y_n + f_h (Y_n) ) h
       + g_h (Y_n)  \Delta W_n 
       + \nu( Y_n ) \Delta N_n,
\end{equation}
where $n \in \{0,1 ,..., N-1\}$ and $ Y_0 = x_0$. Here, by $\Delta W_n := W_{ t_{n+1} } - W_{ t_{n} }$ and $\Delta N_n := N_{ t_{n+1} } - N_{ t_{n} }$ we denote the increments of the Brownian motion and the Poisson process, respectively.
Moreover, the mappings \(f_h\) and \(g_h\) introduced here that depend on time step-size \(h\) are some modifications of $f$ and $g$, respectively.
In what follows, we mainly state the assumptions on
the functions \(f_h\) and \(g_h\).

\begin{assumption}\label{assump:modification_f-g[0.5AS25]}
        Let $f, g$ be defined by 
    \eqref{eq:f_and_g[0.5AS25]}.
    For any $h >0$, for all $x \in D$, there exists a constant $L_1 >0$ such that the mappings
    $f_h$ and $g_h$ obey  
\begin{align}
    \label{eq:assump. f_h(x)<=f(x)[0.5AS25]}
    | f_h(x) | \leq |f(x)| , \quad
    | g_h(x) | &\leq |g(x)| , 
    \\
    \label{eq:assump.f(x) - f_h(x)[0.5AS25]}
    | f(x) - f_h(x) | +
    | g(x) - g_h(x) | 
    &\leq L_1  h^{\frac{1}{2}} ( 1 + | x|^{ 2r } ) .
\end{align}
Moreover, for all $x, y \in D$, there exist constants $L_2,L_3 >0$ and some constant $v>2$ such that
\begin{align}
    \label{eq:assump.f_h(x) - f_h(y)[0.5AS25]} 
    | f_h(x) - f_h(y) | &
    \leq L_2 (1 + h^{-\frac12})  | x - y | , \\ 
    \label{eq:assump.(x-y)(f_h(x) - f_h(y))[0.5AS25]} 
    (x-y)(f_h(x) - f_h(y)) & + \tfrac{v-1}{2} | g_h (x) - g_h (y) | ^{2} 
    \leq L_3 | x-y |^2 .
\end{align}

\end{assumption}

Before proceeding further, we show two examples of $f_h, g_h$ that meet all the above assumptions.

\begin{example}\label{example1}
    Let $f, g$ be defined by 
    \eqref{eq:f_and_g[0.5AS25]}.
    Define the modifications 
    $f_h$ and $g_h$ as
\begin{equation}
f_h(x) := \frac{f(x)}{1 + h^{\frac{1}{2} } |x|^{r}}, \quad g_h(x) := \frac{g(x)}{1 + h^{\frac{1}{2} } |x|^{r}}, 
\quad
 x \in \R.
\end{equation}
\end{example}
Here, the structure of the mappings $f_h$ and $g_h$ is the well-known taming strategy, which was firstly introduced in \cite{hutzenthaler2012strong} and has been extensively studied in many works (see, e.g., \cite{sabanis2016euler,ex_sabanis2013note,liu2023higher,Lord2024convergence,liu2025unconditionally}).

In what follows, we check the conditions in Assumption \ref{assump:modification_f-g[0.5AS25]} one by one.
Note that the condition (\ref{eq:assump. f_h(x)<=f(x)[0.5AS25]}) 
it is obviously fulfilled,
while
the condition (\ref{eq:assump.f(x) - f_h(x)[0.5AS25]}) is also easily confirmed with \( L_1 = \max \{ \alpha_2 , \sigma \} \) by deducing for all $x \in D$,
\begin{align}
      |f(x)-f_h(x)| 
      &
      = \frac{h^{\frac{1}{2}} |x|^r |f(x)|}{1 + h^{\frac{1}{2}} |x|^r} 
      \leq \alpha_2 h^{\frac{1}{2}} |x|^{2r} ,
      \\
       |g(x)-g_h(x)| 
       &
       = \frac{h^{\frac{1}{2}} |x|^{r} |g(x)|}{1 + h^{\frac{1}{2}} |x|^r} 
       \leq
       \sigma h^{\frac{1}{2}}  |x|^{r+\rho} 
       \leq \sigma h^{\frac{1}{2}} (1 + |x|^{2r}) , 
\end{align}
where we used the fact that $\rho \leq \tfrac{r+1}{2} < r$.

Now it remains to verify the conditions 
(\ref{eq:assump.f_h(x) - f_h(y)[0.5AS25]}) and (\ref{eq:assump.(x-y)(f_h(x) - f_h(y))[0.5AS25]}).
We first derive
\begin{equation}\label{eq:derivative_f_h-g_h}
f_h'(x) = \frac{ -\alpha_2 r x^{r-1} }{(1 + h^{\frac{1}{2}} |x|^r)^2}, \quad
g_h'(x) = \frac{\sigma \rho x^{\rho-1} - \sigma (r - \rho) h^{\frac{1}{2}} x^{r+\rho-1}}{(1 + h^{\frac{1}{2}} |x|^r)^2},
\quad
x \in D.
\end{equation}
Since
\begin{equation}
    |f_h'(x)|
    \leq
    \frac{\alpha_2 r ( 1 + |x|^{r}) }{(1 + h^{\frac{1}{2}} |x|^r)^2}
    \leq
    \alpha_2 r ( 1 + h^{-\frac12} ),
\end{equation}
then for all $x,y \in D$,
\begin{equation} \label{example. f_h(x) - f_h(y)[0.5AS25]}
\begin{aligned}
\left| f_h(x) - f_h(y) \right|
=
\left| \int_0^1 f_h'(\theta x + (1-\theta) y) \mathrm{~d}\theta \cdot (x - y) \right|
\leq
\alpha_2 r ( 1 + h^{-\frac12} )|x-y| ,
\end{aligned}
\end{equation}
which confirms (\ref{eq:assump.f_h(x) - f_h(y)[0.5AS25]}) with $L_2=\alpha_2 r$.
Next, we claim for some constants $v >2$ and $L_3>0$ that
\begin{equation}\label{eq:ex1_prove_f'-g'[0.5AS25]}
    f_h'(x) + \tfrac{v-1}{2} |g_h'(x)|^2 \leq L_3,
    \quad
     \text{for all } x \in D.
\end{equation}
To see this, recalling \eqref{eq:derivative_f_h-g_h} one obtains
for all $x \in D$,
\begin{equation}\label{eq:prove_f_h'_beta_g_h'_0.5AS25}
\begin{aligned}
f_h'(x) + \tfrac{v-1}{2} |g_h'(x)|^2 
&= \frac{ -\alpha_2 r x^{r-1} (1 + h^{\frac{1}{2}} x^r)^2 + \tfrac{v-1}{2} [ \sigma \rho x^{\rho-1}-\sigma(r-\rho)h^\frac{1}{2} x^{r+\rho-1}]^2 }
{(1 + h^{\frac{1}{2}} |x|^r)^4} \\
&= \frac{ -\alpha_2 r x^{r-1} - 2 \alpha_2 r h^{\frac{1}{2}} x^{2r-1} - \alpha_2 r h x^{3r-1} }{(1 + h^{\frac{1}{2}} |x|^r)^4} \\
&\quad + \frac{ \tfrac{v-1}{2} [ \sigma^2 \rho^2 x^{2\rho-2} - 2 \sigma^2 \rho (r - \rho) h^{\frac{1}{2}} x^{2\rho + r - 2} + \sigma^2 (r - \rho)^2 h x^{2r + 2\rho - 2} ] }{(1 + h^{\frac{1}{2}} |x|^r)^4}.
\end{aligned}
\end{equation}
Notice that the claim \eqref{eq:ex1_prove_f'-g'[0.5AS25]} immediately follows from \eqref{eq:prove_f_h'_beta_g_h'_0.5AS25} when
$r +1 > 2\rho$.
Moreover,
for the critical case \( r + 1 = 2\rho \), 
by recalling \eqref{eq:parameter_conditions[0.5AS25]}, one can find some $v >2$ such that
$    \tfrac{\alpha_2}{\sigma^2} 
         >
         \tfrac{v-1}{8} \big(r + 2 + \tfrac{1}{r} \big) $,
that is,
\( \alpha_2 r \geq \frac{v - 1}{2} \sigma^2 \rho^2 \).
Then by noting that $2\rho = r+1 > r$,
namely,
$\rho > r- \rho >0$,
one further obtains \( \alpha_2 r \geq \frac{v - 1}{2} \sigma^2 \rho (r - \rho) \) and \( \alpha_2 r \geq \frac{v - 1}{2} \sigma^2 (r - \rho)^2 \).

Hence we derive from \eqref{eq:prove_f_h'_beta_g_h'_0.5AS25} that for $r +1 = 2\rho$,
there exists a constant $L_3$ such that
\begin{equation}
\begin{aligned}
&f_h'(x) + \tfrac{v-1}{2} |g_h'(x)|^2 
\\
&= 
\frac{ -(\alpha_2 r - \tfrac{v-1}{2} \sigma^2 \rho^2 ) x^{r-1}  
-
2 (\alpha_2 r -\tfrac{v-1}{2} \sigma^2 \rho (r - \rho) ) h^{\frac{1}{2}} x^{2r-1} 
-
 (\alpha_2 r - \tfrac{v-1}{2}\sigma^2 (r - \rho)^2  )  h x^{3r-1}
}{(1 + h^{\frac{1}{2}} |x|^r)^4} 
\\
&\leq
L_3.
\end{aligned}
\end{equation}
Thus we conclude the claim 
\eqref{eq:ex1_prove_f'-g'[0.5AS25]} for $r+1 \geq 2\rho$
and it is then deduced
for any $x,y \in D$ that
\begin{equation}
\begin{split}
   &
    (x-y)(f_h(x) - f_h(y)) + \tfrac{v- 1}{2} | g_h (x) - g_h (y) |^{2} 
    \\
  &
    \leq\ 
 |x-y|^2 \left[ \int_0^1 f_h'(\theta x + (1-\theta) y) \, \mathrm{d}\theta
 + \tfrac{ 1}{2} \left| \int_0^1 g_h'(\theta x + (1-\theta) y) \, \mathrm{d}\theta \right|^2 \right] \\
   & \leq\  L_3 | x-y |^2,
\end{split}
\end{equation}
which confirms (\ref{eq:assump.(x-y)(f_h(x) - f_h(y))[0.5AS25]}). As a result, $f_h$ and $g_h$ satisfy all conditions in Assumption \ref{assump:modification_f-g[0.5AS25]}.

\begin{example}\label{example2[0.5AS25]}
    Let $f, g$ be defined as 
    \eqref{eq:f_and_g[0.5AS25]}.
For any given $\kappa \in \big[ \tfrac{1}{2r}$, $\tfrac{1}{2r-2} \big]$, 
we define $f_h$ and $g_h$ by
\begin{equation}
   f_h(x):=f(\P_h(x)),\quad g_h(x):=g(\P_h(x)), \quad x \in \R.
\end{equation}
where $\P_h$ is defined by
\begin{equation}\label{eq:Ph-defn}
    \mathscr{P}_h (x)
    :=
    \min \{ 1 , h^{ - \kappa } |x|^{-1} \} x, \quad x \in \R.
\end{equation}
\end{example}

Here, the projection operator $\mathscr{P}_h$ was introduced by \cite{Beyn_Isaak_Kruse:C-stability_B-consistency_JSC16,Beyn_Isaak_Kruse:C-stability_B-consistency_JSC17} to approximate SDEs in non-globally Lipschitz setting (see the projection Euler/Milstein schemes therein).
In what follows we show that $f_h$ and $g_h$ obey all conditions in Assumption \ref{assump:modification_f-g[0.5AS25]}. 
Firstly, by observing 
\begin{equation}\label{in example_for x-y[0.5AS25]}
    |\mathscr{P}_h(x)| \leq |x|, 
\quad 
    \text{for all } x \in D,
\end{equation}
it is clear that
\begin{equation}
    |f(\mathscr{P}_h(x))| \leq |f(x)|, \quad |g(\mathscr{P}_h(x))| \leq |g(x)|,
    \quad 
    \text{for all } x \in D,
\end{equation}
which shows the condition \eqref{eq:assump. f_h(x)<=f(x)[0.5AS25]}.
Furthermore, one sees
\begin{equation}
    | x - \mathscr{P}_h(x) |
     \leq
       \mathbbm{1}_{\{x \geq h^{ -\kappa } \}} 2 | x |
       \leq
       2 h^\frac12 | x |^{ \frac{1}{2\kappa}  + 1 }  
       \leq
       2 h^\frac12 ( 1 + | x |^{  r + 1 })  ,
       \ \
       \text{for all }  x \in D.
\end{equation}
This implies
\begin{equation}\label{eq:prove_f_x_f_P_h_x}
\begin{aligned}
|f(x) - f(\mathscr{P}_h(x))| &= \left| \int_0^1 f'(\theta x + (1 - \theta) \mathscr{P}_h(x)) \, \mathrm{d}\theta \cdot (x - \P_h(x)) \right| \\
&\leq \left (\int_0^1 |f'(\theta x + (1 - \theta) \P_h(x))| \, \mathrm{d}\theta \right )\cdot |x - \P_h(x)| \\
&\leq \alpha_2 r x^{r - 1} \cdot |x - \P_h(x)| \\
&\leq 2\alpha_2 r h^\frac12 (1+|x|^{2r}),
\end{aligned}
\end{equation}
and similarly,
\begin{equation}\label{eq:prove_g_x_g_P_h_x}
\begin{aligned}
|g(x) - g(\P_h(x))| 
\leq \sigma \rho x^{\rho - 1} \cdot |x - \P_h(x)| 
\leq 2\sigma \rho h^\frac12(1+|x|^{\rho +r})
\leq  2\sigma \rho h^\frac12(1+|x|^{2r}).
\end{aligned}
\end{equation}
Taking $L_1 =\max\{ 2\alpha_2r, 2\sigma \rho \} $ confirms the condition  (\ref{eq:assump.f(x) - f_h(x)[0.5AS25]}).

Next, we notice that
\begin{equation}
\label{eq:P_h_bound}
|\P_h(x) - \P_h(y)| \leq |x - y|, \quad \text{for all } x, y \in D,
\end{equation}
which together with $\mathscr{P}_h(x) \leq h^{ -\kappa}, x \in D$ yields for all $ x, y \in D$,
\begin{equation} \label{example. f(Phi_h(x))-f(Phi_h(y)[0.5AS25]}
\begin{aligned}
\left| f(\mathscr{P}_h(x)) - f(\mathscr{P}_h(y)) \right| &= \left| \int_0^1 f'(\theta \mathscr{P}_h(x) + (1-\theta) \mathscr{P}_h(y)) d\theta \cdot (\mathscr{P}_h(x) - \mathscr{P}_h(y)) \right| \\
&\leq \alpha_2 r \int_0^1 \left| \theta \mathscr{P}_h(x) + (1-\theta) \mathscr{P}_h(y) \right|^{r-1} d\theta \cdot \left| \mathscr{P}_h(x) - \mathscr{P}_h(y) \right| \\
&\leq \alpha_2 r h^{-\kappa(r-1)} |x-y| \\
&\leq \alpha_2 r T^{\frac{1 - \kappa(2 r-2)}{2}} \cdot h^{-\frac12} |x-y|,
\end{aligned}
\end{equation}
where we used the fact that $1 - \kappa(2r-2)  \geq  0$.
This confirms \eqref{eq:assump.f_h(x) - f_h(y)[0.5AS25]} with $L_2 =  \alpha_2 r  T^{\frac{1 - \kappa(2r-2)}{2}}$.

Lastly, by noting that $\mathscr{P}_h(x)$ is a non-decreasing function w.r.t. $x \in D$ and recalling \eqref{eq:P_h_bound}, one infers
\begin{equation}
    (x-y)( \mathscr{P}_h(x) - \mathscr{P}_h(y) ) 
    =
    |x-y| \cdot | \mathscr{P}_h(x) - \mathscr{P}_h(y)  |
    \geq |\mathscr{P}_h(x) - \mathscr{P}_h(y) |^2 ,
\end{equation}
for all $x,y \in D$. This together with the fact $f'(x) \leq 0$ as well as Lemma \ref{lemma:f_g_monotonicity[0.5AS25]} implies

\begin{equation}
\begin{split}
    & (x-y)(f(\mathscr{P}_h(x)) - f(\mathscr{P}_h(y))) + \tfrac{ q-1}{2} | g(\mathscr{P}_h(x)) - g(\mathscr{P}_h(y)) |^{2} \\
    &\leq |\mathscr{P}_h(x) - \mathscr{P}_h(y)|^2 \left[ \int_0^1 f'(\theta \mathscr{P}_h(x) + (1-\theta) \mathscr{P}_h(y)) \, \mathrm{d}\theta \right. \\
    &\quad \left. + \tfrac{ q-1}{2} \left| \int_0^1 g'(\theta \mathscr{P}_h(x) + (1-\theta) \mathscr{P}_h(y)) \, \mathrm{d}\theta \right|^2 \right] \\
    &\leq L | \mathscr{P}_h(x) - \mathscr{P}_h(y) |^2 \\
    &\leq L | x - y |^2,
\end{split}
\end{equation}
where $q > 2$ comes from Lemma \ref{lemma:f_g_monotonicity[0.5AS25]}.
This verifies (\ref{eq:assump.(x-y)(f_h(x) - f_h(y))[0.5AS25]}) with $v=q$ and $L_3=L$.

Hence
it is concluded that the modification functions $f_h$ and $g_h$ satisfy all conditions in Assumption \ref{assump:modification_f-g[0.5AS25]}.

~

Note that by solving a quadratic equation,
the unique positive solution determined by (\ref{eq:numeri_method[0.5AS25]}) can be explicitly given as
\begin{equation}\label{eq:solution_of_scheme}
    Y_{n+1} = \frac{1}{2} \left[ Y_n + \vartheta_n h + \mathcal{S}_n
    +
    \sqrt{\left( Y_n + \vartheta_n h + \mathcal{S}_n\right)^2 + 4\alpha_{-1} h} \right] > 0,  
\end{equation}
where for short we denote
\begin{equation}
   \vartheta_n := -\alpha_0 + \alpha_1 Y_n + f_h(Y_n),\quad \mathcal{S}_n := g_h(Y_n) \Delta W_n + \nu( Y_n ) \Delta N_n.
\end{equation}
Consequently, the scheme admits the following well-posedness and positivity-preserving property.
\begin{lem}\label{lem:wellposed_of_Numerical method[0.5AS25]}    
     For any step-size $h = \tfrac{T}{N} > 0$,
     the numerical scheme \eqref{eq:numeri_method[0.5AS25]} is well-posed and positivity preserving, 
     i.e., it admits a unique positive $\{ \mathcal{F}_{t_n}\}_{n=0}^N $-adapted solution $\{ Y_n \}_{ n=0 }^{N}, N \in \mathbb{N} $ for the scheme \eqref{eq:numeri_method[0.5AS25]} given a positive initial data. 
\end{lem}

\section{The mean-square convergence of the scheme}\label{section:error analysis[0.5AS25]}

Our aim in this section is to establish 
the mean-square convergence analysis for the proposed scheme 
\eqref{eq:numeri_method[0.5AS25]}.
To start with,
it is convenient to recast 
\eqref{eq:ait-sahalia[0.5AS25]} as: 
\begin{equation}\label{eq:ait_sahalia_gridpointVersion[0.5AS25]}
\begin{split}
        X_{ t_{n+1} }
        = 
         X_{ t_{n} }
      & + 
        ( \alpha_{-1} X_{ t_{n+1} }^{-1} - \alpha_0 + \alpha_1  X_{ t_{n} } ) + f_h( X_{ t_{n} } )   ) h  
       + 
        g_h(X_{ t_{n} }) \Delta W_n
        +
        \nu( X_{ t_{n} } ) \Delta N_n
        + 
        R_{n+1}, 
\end{split}
\end{equation}
for $n \in \{ 0,1,..., N-1 \}$, $N \in \mathbb{N}$, where we denote
\begin{equation}\label{R_i[0.5AS25]}
\begin{split}
    R_{n+1} 
    := 
  & 
    \int_{t_n}^{ t_{n+1} } 
    [ \alpha_{-1} X_{s}^{-1} - \alpha_{-1} X_{ t_{n+1} }^{-1} + \alpha_1 X_s -  \alpha_1  X_{ t_{n} } + f(X_s) - f_h( X_{ t_{n} } ) ] 
    \,\mathrm{d}s \\
    & \quad 
    + 
    \int_{t_n}^{ t_{n+1} }
    [  g(X_s) - g_h( X_{ t_{n} } )  
        ] \,\mathrm{d} W_s
 + 
    \int_{t_n}^{ t_{n+1} }
    [  \nu(X_{s^-}) -  \nu( X_{ t_{n} } )  
        ] \,\mathrm{d} N_s .
\end{split}
\end{equation}

With these preparations, we are ready to state the 
upper mean-square error bounds for the underlying schemes.

\begin{thm}[Upper mean-square error bounds]\label{thm(main):Order_onehalf[0.5AS25]}
    Let $\{ X_t \}_{ t\in [0,T] }$ and $\{ Y_n \}_{n=0}^N $ be the solutions of \eqref{eq:ait-sahalia[0.5AS25]} and \eqref{eq:numeri_method[0.5AS25]}, respectively. Let Assumption 
    \ref{assump:modification_f-g[0.5AS25]} 
    hold. 
Then there exists a constant $C$ independent of $h$ such that for all $n \in \{1,2,...,N\}$,
\begin{equation}
    \| X_{ t_{n} } - Y_n \|_{L^2(\Omega;\mathbb{R})} 
    \leq  
   C
   \Big(
   h^{-\frac{1}{2}}
    \sup_{i \in \{1,2,...,n\} }
   \left\| R_{i} \right\|_{L^2(\Omega;\mathbb{R})} 
 +
    h^{-1} 
    \sup_{i \in \{1,2,...,n\} }
    \left\| \mathbb{E}[R_{i}|\mathcal{F}_{t_{i-1}}] \right\|_{L^2(\Omega;\mathbb{R})} 
    \Big)
.
\end{equation}
\end{thm}

\begin{proof}
For brevity, we introduce the following notation: 
for all $k \in \{ 0,1,2, ..., N \}$,
\begin{equation}\label{denotion:Delta_f[0.5AS25]}
\begin{split}
    & e_k :=  X_{ t_{k} } - Y_k, 
     \quad 
    \Delta f_{h,k}^{X ,Y} := f_h(  X_{ t_{k} }) -f_h(Y_k),
    \\
    &
    \Delta g_{h,k}^{ X ,Y} := g_h(  X_{ t_{k} }) -g_h(Y_k), \quad
    \Delta \nu_k^{X,Y} := \nu( X_{ t_{k} } ) - \nu( Y_k ).
\end{split}
\end{equation}
For all $n \in \{0, 1,2, ...,N-1\}$, 
subtracting \eqref{eq:numeri_method[0.5AS25]} from \eqref{eq:ait_sahalia_gridpointVersion[0.5AS25]} gives 
\begin{equation}\label{eq:e_n[0.5AS25]}
\begin{split}
    e_{n+1} - h \cdot \alpha_{-1}  (  X_{ t_{n+1} }^{-1} - Y_{n+1}^{-1}  ) 
    & =
    ( 1 + h \alpha_{1} ) e_n
     +
     h \Delta f_{h,n}^{ X, Y}   
     +
     \Delta g_{h,n}^{X, Y}  \Delta W_n
    + 
     \Delta \nu_n^{X, Y}  \Delta N_n
    +
    R_{n+1}.  
\end{split}
\end{equation}
Note that
\begin{equation}
\begin{aligned}
    & \vert e_{n+1} - h \cdot \alpha_{-1}  (  X_{ t_{n+1} }^{-1} - Y_{n+1}^{-1}  ) \vert^2 \\
    & =
    \vert e_{n+1} \vert^2
   +
    2 h \alpha_{-1} | e_{n+1} |^2  \cdot
        \int_0^1 |Y_{n+1} + \theta(X_{t_{n+1}} - Y_{n+1})|^{-2}  \,\mathrm{d}\theta
    +
    h^2 \alpha_{-1}^2 | X_{ t_{n+1} }^{-1} - Y_{n+1}^{-1}  |^2\\
    & \geq
    \vert e_{n+1} \vert^2.
\end{aligned}
\end{equation}
Thus squaring both sides of equation \eqref{eq:e_n[0.5AS25]} and taking expectations yield
\begin{equation}\label{eq:square_of_e_n+1,1[0.5AS25]}
\begin{split}
   \E \Big[ \big| e_{n+1} \big|^2 \Big] 
    & \leq 
    ( 1 + h \alpha_{1} )^2  \mathbb{E}\Big [ | e_n |^2  \Big]   
     +  
     h^2   \mathbb{E}\Big [|  \Delta f_{h,n}^{ X, Y}   |^2 \Big]   
    +
       \mathbb{E}\Big [ |\Delta g_{h,n}^{ X, Y} \Delta W_n   |^2 \Big]      
    +
    \mathbb{E}\Big [ |  \Delta \nu_n^{X, Y}  \Delta N_n |^2 \Big] \\   
    &\quad
    +  
   \mathbb{E}\Big [ | R_{n+1} |^2\Big]   
          + 
          2 h( 1 + h \alpha_{1} )  \mathbb{E}\Big [e_n \Delta f_{h,n}^{ X, Y}  \Big]  
          +  
          2 ( 1 + h \alpha_{1} )  \mathbb{E}\Big [e_n  \Delta g_{h,n}^{ X, Y} \Delta W_n \Big] \\   
          & \quad
          +  
          2 ( 1 + h \alpha_{1} ) \mathbb{E}\Big [e_n  \Delta \nu_n^{X, Y}  \Delta N_n \Big]  
          +  
          2 ( 1 + h \alpha_{1} ) \mathbb{E}\Big [ e_n  R_{n+1} \Big] \\    
              & \quad
              + 
              2 h \mathbb{E}\Big [\Delta f_{h,n}^{X, Y} \Delta g_{h,n}^{X, Y} \Delta W_n\Big]     
              + 
              2 h \mathbb{E}\Big [\Delta f_{h,n}^{X, Y}  \Delta \nu_n^{X, Y}  \Delta N_n \Big]   
              + 
              2 h \mathbb{E}\Big [\Delta f_{h,n}^{X, Y}  R_{n+1} \Big]\\   
                   & \quad
                   +  
                   2 \mathbb{E}\Big [\Delta g_{h,n}^{X, Y} \Delta W_n   \Delta \nu_n^{X, Y}  \Delta N_n \Big]  
                   + 
                   2 \mathbb{E}\Big [\Delta g_{h,n}^{X, Y} \Delta W_n  R_{n+1} \Big] \\  
                         & \quad
                         +  2  \mathbb{E}\Big [\Delta \nu_n^{X, Y}  \Delta N_n R_{n+1}\Big]. 
    \\ & =   
    ( 1 + h \alpha_{1} )^2  \mathbb{E}\Big [ | e_n |^2  \Big]   
     +  
     h^2   \mathbb{E}\Big [|  \Delta f_{h,n}^{ X, Y}   |^2 \Big]   
    +
      h \mathbb{E}\Big [ |\Delta g_{h,n}^{ X, Y}   |^2 \Big]    
    +
   (\lambda h+\lambda^2 h^2) \mathbb{E}\Big [ |  \Delta \nu_n^{X, Y}   |^2 \Big] \\   
    &\quad
    +  
   \mathbb{E}\Big [ | R_{n+1} |^2\Big]  
          + 
          2 h( 1 + h \alpha_{1} )  \mathbb{E}\Big [e_n \Delta f_{h,n}^{ X, Y}  \Big]  
          +  
          2 \lambda h( 1 + h \alpha_{1} ) \mathbb{E}\Big [e_n  \Delta \nu_n^{X, Y}   \Big] 
          \\
           & \quad +  
          2 ( 1 + h \alpha_{1} ) \mathbb{E}\Big [ e_n  R_{n+1} \Big] 
              + 
              2 \lambda h^2 \mathbb{E}\Big [\Delta f_{h,n}^{X, Y}  \Delta \nu_n^{X, Y}  \Big] 
              + 
              2 h \mathbb{E}\Big [\Delta f_{h,n}^{X, Y}  R_{n+1} \Big]   
                  \\
             &\quad
                   + 
                   2 \mathbb{E}\Big [\Delta g_{h,n}^{X, Y} \Delta W_n  R_{n+1} \Big] 
                         +  2  \mathbb{E}\Big [\Delta \nu_n^{X, Y}  \Delta N_n R_{n+1}\Big], 
\end{split}
\end{equation}
Here, the last equality holds due to the facts that
\begin{equation}
     \E \Big[ \big| \Delta W_n \big|^2 \Big]=h,
     \ \
     \E \Big[ \big| \Delta N_n \big|^2 \Big]
     =
     \E [ | \Delta \Tilde{N}_n |^2 ]
          +
          | \lambda h |^2
          +
          2 \E [ \lambda h  \Delta \Tilde{N}_n ]
    =
    \lambda h + \lambda^2 h^2 
     ,
\end{equation}
where we denote $\Delta \Tilde{N}_n:= \Tilde{N}_{n+1}-\Tilde{N}_n$,
and additionally,
\begin{equation*}
    \begin{split}
    &  
         \E \Big[ e_n \Delta g_{h,n}^{ X, Y} \Delta W_n \Big] 
         = 0 ,\quad
     \E \Big[ \Delta f_{h,n}^{ X, Y} \Delta g_{h,n}^{ X, Y} \Delta W_n \Big] 
         = 0,
        \quad
          \E\Big[ \Delta g_{h,n}^{ X, Y} \Delta W_n   \Delta \nu_n^{X, Y} \Delta N_n   \Big] = 0,
    \end{split}
\end{equation*}
together with
\begin{align*}
        & 
        \E \Big[ \big|\Delta g_{h,n}^{X, Y} \Delta W_n \big|^2 \Big] 
        =
        \E \Big[ \big|\Delta g_{h,n}^{X, Y}  \big|^2 \Big]
        \cdot
        \E \Big[ \big| \Delta W_n \big|^2 \Big]
        = 
        h \E \Big[ \big|\Delta g_{h,n}^{X, Y} \big|^2 \Big] , 
\\
    &
    \E \Big[ \big|\Delta \nu_n^{X, Y} \Delta N_n \big|^2 \Big] 
    =
    \E \Big[ \big|\Delta \nu_n^{X, Y} \big|^2 \Big]
        \cdot
        \E [ | \Delta N_n |^2 ]
    =
     ( \lambda h + \lambda^2 h^2 ) 
       \E \Big[ \big|\Delta \nu_n^{X, Y} \big|^2 \Big],
       \\
&
  \E \Big[  e_n  \Delta \nu_n^{X, Y} \Delta N_n  \Big]
  =
    \E \Big[ e_n  \Delta \nu_n^{X, Y} \Delta \Tilde{N}_n  \Big]
    +
    \lambda h \E \Big[  e_n  \Delta \nu_n^{X, Y} \Big]
    =
    \lambda h \E \Big[  e_n  \Delta \nu_n^{X, Y} \Big],
    \\
    &
   \E \Big[ \Delta f_{h,n}^{X, Y} \Delta \nu_n^{X, Y} \Delta N_n  \Big] 
 =
   \E \Big[ \Delta f_{h,n}^{X, Y} \Delta \nu_n^{X, Y} \Delta \Tilde{N}_n  \Big] 
   +
   \lambda h \E \Big[ \Delta f_{h,n}^{X, Y} \Delta \nu_n^{X, Y} \Big] 
   =
   \lambda h \E \Big[ \Delta f_{h,n}^{X, Y} \Delta \nu_n^{X, Y} \Big],
\end{align*}
by observing that the terms $\Delta f_{h,n}^{X, Y}$, $\Delta g_{h,n}^{X, Y}$ and $\Delta \nu_n^{X, Y} $ are $\mathcal{F}_{t_n}$-measurable and are therefore independent of $\Delta W_n$ and $\Delta \tilde{N}_n$ for all $n \in \{0, 1,2, ...,N-1\}$.
Further, the Young inequality implies that
    \begin{align}
        2 h^2\alpha_{1}
        \E\Big[ e_n  \Delta f_{h,n}^{ X, Y} \Big]
        & \leq 
        h^2 \alpha_1^2\E\Big[ | e_n |^2 \Big]
          +
        h^2 \E\Big[ | \Delta f_{h,n}^{X, Y} |^2 \Big] , 
        \\
         2 \lambda h( 1 + h \alpha_{1} ) \mathbb{E}\Big [e_n  \Delta \nu_n^{X, Y}   \Big]
         &\leq
         \lambda h( 1 + h \alpha_{1} ) \mathbb{E}\Big [ |e_n |^2  \Big]
         +
          \lambda h( 1 + h \alpha_{1} ) \mathbb{E}\Big [| \Delta \nu_n^{X, Y}  |^2 \Big],
        \\
          2 \lambda h^2 \mathbb{E}\Big [\Delta f_{h,n}^{X, Y}  \Delta \nu_n^{X, Y}  \Big] 
          &\leq
          \lambda h^2 \mathbb{E}\Big [ |\Delta f_{h,n}^{X, Y} |^2 \Big] 
          +
          \lambda h^2 \mathbb{E}\Big [| \Delta \nu_n^{X, Y} |^2 \Big], 
    \end{align}
and
    \begin{align}  
        2 h \E\Big[ \Delta f_{h,n}^{X, Y}  R_{n+1} \Big] 
        & \leq
        h^2 \E\Big[ | \Delta f_{h,n}^{X, Y} |^2 \Big]
          +
        \E\Big[ | R_{n+1} |^2 \Big] ,
         \\
        2 \E\Big[ \Delta \nu_n^{X, Y} \Delta N_n  R_{n+1} \Big]
        &
        \leq
        ( \lambda h + \lambda^2 h^2 )
        \E \Big[ \big| \Delta \nu_n^{X, Y} \big|^2 \Big]
        +
        \E\Big[ | R_{n+1} |^2 \Big] ,
        \\
        2 \E\Big[ \Delta g_{h,n}^{ X, Y} \Delta W_n  R_{n+1} \Big]
        & \leq
        h (v-2) \E\Big[ | \Delta g_{h,n}^{ X, Y} |^2 \Big]
          +\tfrac{1}{v-2}
\E\Big[ |R_{n+1}|^2 \Big] ,
    \end{align}
where $v>2$ comes from \eqref{eq:assump.(x-y)(f_h(x) - f_h(y))[0.5AS25]}.
In view of the above estimates and the condition
\eqref{eq:nu_assmption1[0.5AS25]}, one immediately derives from \eqref{eq:square_of_e_n+1,1[0.5AS25]} that
\begin{equation}\label{E(square of V_j+1)[0.5AS25]}
\begin{split}
     \E\big[ | e_{n+1} |^2 \big] 
&\leq
    \big( 
       ( 1 + h \alpha_{1}  )^2
       +
       h^2 \alpha_{1}^2 
       +
       \lambda h( 1 + h \alpha_{1} )
    \big)
    \mathbb{E}\Big [ | e_n |^2  \Big] 
     +  
     (3 + \lambda) h^2
     \mathbb{E}\Big [|  \Delta f_{h,n}^{ X, Y}   |^2 \Big]  
 \\
    &\quad
    +
      (v-1)h 
      \mathbb{E}\Big [ |\Delta g_{h,n}^{ X, Y}   |^2 \Big]    
    +
   [2\lambda h + 2\lambda^2 h^2 + \lambda h^2 + \lambda h( 1 + h \alpha_{1} ) ]
   \mathbb{E}\Big [ |  \Delta \nu_n^{X, Y}   |^2 \Big] \\   
    &\quad
    +  
   (3 + \tfrac{1}{v-2}) \mathbb{E}\Big [ | R_{n+1} |^2\Big]   
          + 
          2 h  \mathbb{E}\Big [e_n \Delta f_{h,n}^{ X, Y}  \Big]  
          +  
          2 ( 1 + h \alpha_{1} ) \mathbb{E}\Big [ e_n  R_{n+1} \Big] 
\\
    &  \leq 
    \big[( 1 + h \alpha_{1} )^2  +  h^2 \alpha_1^2 + \lambda h (1 + h \alpha_1) 
    +
    ( 3\lambda h + 2\lambda^2 h^2 
     +
     \lambda h^2 
     +
     \alpha_1 \lambda h^2  )
     M^2
    \big]   
     \E\Big[ \big| e_n \big|^2 \Big]   
     \\
     & \quad
     +
     (3 + \lambda) h^2  \E\Big[ \big|  \Delta f_{h,n}^{ X, Y}  \big|^2 \Big]
     +
      2 h \bigg(\tfrac{v-1}{2}
\E\Big[ \big|\Delta g_{h,n}^{ X, Y} \big|^2 \Big]  
         +
         \E\Big[ e_n \Delta f_{h,n}^{ X, Y} \Big]  
     \bigg) \\
     & \quad
     +
     (3+\tfrac{1}{v-2}) \E\Big[ \big| R_{n+1} \big|^2 \Big]  
    +  
    2 ( 1 + h \alpha_{1} ) 
    \E\Big[ e_n  R_{n+1} \Big].  
\end{split}
\end{equation}
Using \eqref{eq:assump.f_h(x) - f_h(y)[0.5AS25]} and \eqref{eq:assump.(x-y)(f_h(x) - f_h(y))[0.5AS25]} leads to
\begin{equation}\label{iterateion[0.5AS25]}
\begin{aligned}
    \E\big[ | e_{n+1} |^2 \big] 
& 
\leq
     \big[( 1 + h \alpha_{1} )^2  +  h^2 \alpha_1^2 + \lambda h (1 + h \alpha_1) 
    +
    ( 3\lambda h + 2\lambda^2 h^2 
     +
     \lambda h^2 
     +
     \alpha_1 \lambda h^2  )
     M^2
    \big]   
     \E\Big[ \big| e_n \big|^2 \Big] 
\\
   &\quad +
    [ (3+\lambda) h^2 L_2^2 (1 + h^{-\frac12})^2 + 2 L_3 h] \E\Big[ \big| e_n \big|^2 \Big]
\\
     & \quad
     +
    (3+\tfrac{1}{v-2}) \E\Big[ \big| R_{n+1} \big|^2 \Big]  
    +  
    2 ( 1 + h \alpha_{1} ) 
    \E\Big[ e_n  R_{n+1} \Big]
\\
& \leq
(1 + Ch )\E\Big[ \big| e_n \big|^2 \Big]     
     +
      (3+\tfrac{1}{v-2})\E\Big[ \big| R_{n+1} \big|^2 \Big]  
    +  
    2 ( 1 + h \alpha_{1} ) 
    \E\Big[ 
        e_n 
        \cdot
        \E\big[ R_{n+1}| \mathcal{F}_{t_n} \big]  
    \Big],
\end{aligned}
\end{equation}
where we used the property of the conditional expectation in the second equality.
Further, utilizing the Young inequality gives
\begin{equation}
\begin{aligned}
    \E\big[ | e_{n+1} |^2 \big] 
    & \leq
   [1 + Ch +  h ( 1 + h \alpha_{1} )  ] 
     \E\Big[ \big| e_n \big|^2 \Big]   
     +
      (3+\tfrac{1}{v-2}) \E\Big[ \big| R_{n+1} \big|^2 \Big]  
    +
    \tfrac{1 + h \alpha_{1}}{h}
    \E\Big[ 
        \big\vert \E\big[ R_{n+1}| \mathcal{F}_{t_n} \big] \big\vert^2
    \Big] . 
\end{aligned}
\end{equation}
By iteration and observing that $e_0 = 0$, one finally obtains
\begin{equation}
\label{eq:relationship_error}
\begin{aligned}
    \E \big[ \vert e_{n+1} \vert^2 \big]
    & \leq
    (1+ C h)^{ n + 1} \E\big[ | e_0 |^2 \big]
\\
&\quad
    +
     \sum_{i=0}^{ n } (1+ C h)^{n-i}
  \left\{
   (3+\tfrac{1}{v-2})  \E\Big[ \big| R_{i+1} \big|^2 \Big]  
    +
   \tfrac{1 + h \alpha_{1}}{h} \E\Big[ 
        \big\vert \E\big[ R_{i+1}| \mathcal{F}_{t_i} \big] \big\vert^2\Big]
   \right\}
   \\
    & 
    \leq 
    C h^{-1}
    \left\{
   (3+\tfrac{1}{v-2})  
   \sup_{i \leq n+1, i \in \N}
   \E\Big[ \big| R_{i} \big|^2 \Big]  
    +
   \tfrac{1 + h \alpha_{1}}{h} 
   \sup_{i \leq n+1, i \in \N} 
   \E\Big[ 
        \big\vert \E\big[ R_{i}| \mathcal{F}_{t_{i-1} } \big] \big\vert^2\Big]
   \right\}.
\end{aligned}
\end{equation}
The desired result then follows.
\end{proof}

To obtain the convergence rate, it remains to bound the terms \(\|R_i\|_{L^2(\Omega; \mathbb{R})}\) and \(\|\mathbb{E}(R_i| \mathcal{F}_{t_{i-1}})\|_{L^2(\Omega; \mathbb{R})}\), \(i \in \{1, 2, \ldots, N\}\) 
 appearing in Theorem \ref{thm(main):Order_onehalf[0.5AS25]}.
These estimates are presented in the following lemma.

\begin{lem}\label{lem. R_n estimates [0.5AS25]}
    Let $\{ X_t \}_{ t\in [0,T] }$ and $\{ Y_n \}_{n=0}^N $ be the solutions of \eqref{eq:ait-sahalia[0.5AS25]} and \eqref{eq:numeri_method[0.5AS25]}, respectively. Let Assumption \ref{assump:modification_f-g[0.5AS25]} hold. 
    If one of the following conditions stands:
\begin{enumerate}[(1)]
    \item $r + 1 > 2 \rho$,
    \item \label{enumerate:theorem_critical_case[0.5AS25]}
    $r + 1 = 2 \rho$, $\tfrac{\alpha_2}{\sigma^2} \geq 2r - \tfrac12$,
\end{enumerate}
    then there exists a constant $C$ independent of $h$ such that for all $n \in \{0,1,...,N-1 \} $, $N \in \N$,
    \begin{equation}
         \| R_{n+1} \|_{L^2(\Omega;\R)}
        \leq
        C h,
        \quad
        \| \E [ R_{n+1} | \F_{t_n} ] \|_{L^2(\Omega;\R)}
        \leq
        C h^{\frac{3}{2}}.
    \end{equation}
\end{lem}

\begin{proof}

By the Minkowski inequality, we split $\| R_i \|_{L^2(\Omega;R)},i\in \{1,2,...,N\}$ into three terms as
  \begin{equation}\label{eq:R_i(r+1>2rho)[0.5AS25]}
    \begin{split}
        \| R_i \|_{L^2(\Omega;\R)}
        & \leq
         \underbrace{
         \bigg\|   
                 \int_{t_{i-1}}^{t_i}  
                        [ \alpha_{-1} X_{s}^{-1} - \alpha_{-1}   X_{ t_{i} } ^{-1} + \alpha_1 X_s -  \alpha_1  X_{ t_{i-1} } + f(X_s) - f_h(  X_{ t_{i-1} } ) ]  \mathrm{~d}s  
         \bigg\|_{L^2(\Omega;\R)} 
                                    }_{ = : I_1 }   \\
        & \qquad
         +   
         \underbrace{
         \bigg\|   
                  \int_{t_{i-1}}^{t_i} 
                     [ g(X_s) - g_h(  X_{ t_{i-1} } )  ] \mathrm{~d}W_s  
         \bigg\|_{L^2(\Omega;\R)}   
                                       }_{ = : I_2 }
          +   
         \underbrace{
         \bigg\|   
                  \int_{t_{i-1}}^{t_i} 
                     [ \nu(X_{s^-}) - \nu(  X_{ t_{i-1} } ) ] \mathrm{~d}N_s  
         \bigg\|_{L^2(\Omega;\R)}   
                                       }_{ = : I_3 }     .
    \end{split}
    \end{equation}
Under Assumption \ref{assump:modification_f-g[0.5AS25]}, the first term $I_1$ can be estimated as
\begin{equation}\label{eq:I_1[0.5AS25]}
    I_1
    \leq
  C h
    \Big(   
    1 
    +
    \sup_{ t \in [0,T] } 
    \Vert X_t^{-1} \Vert_{ L^{ 2}(\Omega;\R) }
    +
    \sup_{ t \in [0,T] } \| X_t \|^{ r }_{ L^{ 2 r }(\Omega;\R) }   \Big).   
\end{equation}
For the term $I_{3}$,
employing the Young inequality, the It\^o isometry, the inequality
\eqref{eq:nu_assmption1[0.5AS25]} and Lemma \ref{lem:Holder_continuous_X[0.5AS25]} gives
\begin{equation}\label{eq:I_3[0.5AS25]}
    \begin{split}
        |I_{3}|^2
        &
        =
        \bigg\|   
                  \int_{t_{i-1}}^{t_i} 
                     [ \nu(X_{s^-}) - \nu(  X_{ t_{i-1} } ) ] 
               \mathrm{~d}( \Tilde{N}_s + \lambda s)  
         \bigg\|^2_{L^2(\Omega;\R)}   
    \\
        & \leq
        2
        \bigg\|   
                  \int_{t_{i-1}}^{t_i} 
                     [ \nu(X_{s^-}) - \nu(  X_{ t_{i-1} } ) ] 
               \mathrm{~d} \Tilde{N}_s  
         \bigg\|^2_{L^2(\Omega;\R)} 
         +
         2 \lambda^2
         \bigg\|   
                  \int_{t_{i-1}}^{t_i} 
                     [ \nu(X_{s^-}) - \nu(  X_{ t_{i-1} } ) ] 
               \mathrm{~d}s
         \bigg\|^2_{L^2(\Omega;\R)} 
     \\
         & \leq
        ( 
        2 \lambda
        +
        2 \lambda^2 h
        )
        \int_{t_{i-1}}^{t_i}
         \| \nu(X_{s^-}) - \nu( X_{ t_{i-1} } )  \|^2_{L^2(\Omega;\R)}
         \mathrm{~d}s    
    \\
         & \leq
        C M^2
        \int_{t_{i-1}}^{t_i}
         \| X_{s^-} -  X_{ t_{i-1} } \|^2_{L^2(\Omega;\R)}
         \mathrm{~d}s   
    \\
    &  \leq
    C h^2.
    \end{split}
\end{equation}
A more delicate estimate is required for the term $I_2$.
To this end, we introduce a simplified notation
\begin{equation}
    F(x) := \alpha_{-1} x^{-1} - \alpha_0 + \alpha_1 x - \alpha_2 x^r, 
    \ \
    x \in \mathbb{R}.
\end{equation}
By adding and subtracting the terms 
$g(X_{t_{i-1}})$ to the integrand of $I_2$, 
and making use of the It\^o isometry
as well as the Young inequality one gets
\begin{equation}
    |I_{2}|^2
        \leq
        2 \int_{t_{i-1}}^{t_i}
         \| g(X_s) - g( X_{ t_{i-1} } )  \|^2_{L^2(\Omega;\R)}
         \mathrm{~d}s
       +
         2 \int_{t_{i-1}}^{t_i}
         \| g(  X_{ t_{i-1} } ) - g_h( X_{ t_{i-1} })
            \|^2_{L^2(\Omega;\R)} 
        \mathrm{~d}s  .    
\end{equation}
The generalized It\^o formula \cite{gardon2004order} and \eqref{eq:assump.f(x) - f_h(x)[0.5AS25]} then lead to
\begin{equation}
    \begin{split}
        |I_{2}|^2
& \leq
        2 \int_{t_{i-1}}^{t_i}
        \bigg\| 
        \int_{t_{i-1}}^{s} g'(X_l) F(X_l) \mathrm{~d}l
        +
        \int_{t_{i-1}}^{s} g'(X_l) g(X_l) \mathrm{~d}W_l
        +
        \frac12
        \int_{t_{i-1}}^{s}
         g''(X_l) g(X_l)^2 \mathrm{~d}l
\\
&\quad
        +
        \int_{t_{i-1}}^{s}
        [ g( X_{l^-} + \nu(X_{l^-})) - g(X_{l^-} ) ] \mathrm{~d} N_l
        \bigg\|^2_{L^2(\Omega;\R)}
         \mathrm{~d}s
         +
         4 L_1^2 h^2
       \big( 1 
         +
         \|X_{ t_{i-1} }\|^{4r}_{L^{4r}(\Omega;\R)} 
         \big).
\end{split}
\end{equation}
Another application of the Young inequality gives
\begin{equation}
    \begin{split}
        |I_{2}|^2
&\leq
        8 \int_{t_{i-1}}^{t_i}
        \bigg\| 
        \int_{t_{i-1}}^{s} g'(X_l) F(X_l) \mathrm{~d}l
        \bigg\|^2_{L^2(\Omega;\R)} 
        \mathrm{~d}s
        +
        8 \int_{t_{i-1}}^{t_i}
        \bigg\| 
        \int_{t_{i-1}}^{s} g'(X_l) g(X_l) \mathrm{~d}W_l
        \bigg\|^2_{L^2(\Omega;\R)} 
        \mathrm{~d}s
\\
&\quad
        +
        2 \int_{t_{i-1}}^{t_i}
        \bigg\|
        \int_{t_{i-1}}^{s}
         g''(X_l) g(X_l)^2 \mathrm{~d}l
        \bigg\|^2_{L^2(\Omega;\R)} 
        \mathrm{~d}s
\\
&\quad
        +
        8 \int_{t_{i-1}}^{t_i}
        \bigg\|
        \int_{t_{i-1}}^{s}
        [ g( X_{l^-} + \nu(X_{l^-})) - g(X_{l^-} ) ] \mathrm{~d} ( \Tilde{N}_l + \lambda l )
        \bigg\|^2_{L^2(\Omega;\R)} 
        \mathrm{~d}s
\\
&\quad
         +
         4 L_1^2 h^2
       \big( 1 
         +
         \|X_{ t_{i-1} }\|^{4r}_{L^{4r}(\Omega;\R)} \big).
\end{split}
\end{equation}

Furthermore, in view of the H\"older inequality, 
it follows that
\begin{equation}\label{eq:I_2[0.5AS25]}
    \begin{split}
|I_2|^2
&\leq
        8 h \int_{t_{i-1}}^{t_i}
        \int_{t_{i-1}}^{s} 
        \| 
        g'(X_l) F(X_l) 
        \|^2_{L^2(\Omega;\R)} 
        \mathrm{~d}l \mathrm{d}s
        +
        8 \int_{t_{i-1}}^{t_i}
        \int_{t_{i-1}}^{s} 
        \| 
        g'(X_l) g(X_l)
        \|^2_{L^2(\Omega;\R)} 
        \mathrm{~d} l \mathrm{d}s
\\
&\quad
        +
        2 h \int_{t_{i-1}}^{t_i}
        \int_{t_{i-1}}^{s}
        \|
         g''(X_l) g(X_l)^2 
        \|^2_{L^2(\Omega;\R)} 
        \mathrm{~d}l \mathrm{d}s
\\
&\quad
        +
        16 \lambda \int_{t_{i-1}}^{t_i}
        \int_{t_{i-1}}^{s}
        \|
        g( X_{l^-} + \nu(X_{l^-})) - g(X_{l^-} ) 
        \|^2_{L^2(\Omega;\R)} 
        \mathrm{~d} l \mathrm{d}s
\\
&\quad
        +
        16 \lambda^2 h \int_{t_{i-1}}^{t_i}
        \int_{t_{i-1}}^{s}
        \|
        g( X_{l^-} + \nu(X_{l^-})) - g(X_{l^-} ) 
        \|^2_{L^2(\Omega;\R)} )
        \mathrm{~d}l \mathrm{d}s
\\
&\quad
         +
         4 L_1^2 h^2
       \big( 1 
         +
         \|X_{ t_{i-1} }\|^{4r}_{L^{4r}(\Omega;\R)} 
         \big)
\\
&\leq
C h^2 \Big(1 + \mathbbm{1}_{ \{ \rho<2\} } \sup_{t \in [0,T]} \|X_{t}^{-1}\|^{4-2\rho}_{L^{4-2\rho}(\Omega;\R)} + \sup_{t \in [0,T]} \|X_{t}\|^{4r}_{L^{4r}(\Omega;\R)} \Big),
    \end{split}
\end{equation}
where we utilized the polynomial growth of $g$ and the linear growth condition of $\nu$ \eqref{eq:nu_assmption3[0.5AS25]} in the last inequality.
A combination of \eqref{eq:I_1[0.5AS25]},  \eqref{eq:I_3[0.5AS25]} and \eqref{eq:I_2[0.5AS25]} 
together with the Lyapunov inequality yields
\begin{equation}\label{eq:R_i_L2_estimate[0.5AS25]}
    \| R_i \|_{L^2(\Omega;\R)} 
    \leq
      C h
      \Big( 
          1 
       +
         \Vert X_t^{-1} \Vert_{ L^{ 2 }(\Omega;\R) }
       +
         \sup_{ t \in [0,T] } \| X_t \|^{ 2r }_{L^{ 4r }(\Omega;\R)}
       \Big)
    \leq Ch,
\end{equation}
where we employed Lemma \ref{lem:moment_case1[0.5AS25]} due to the fact that $4r \leq \tfrac{\sigma^2 + 2 \alpha_2}{\sigma^2}$, which is equivalent to
$\tfrac{\alpha_2}{\sigma^2} \geq 2r - \tfrac12$ and guarantees the required moment bound for the general critical case $r + 1 = 2 \rho$.

Now we turn to the estimations for $\| \E (R_i| \F_{i-1}) \|_{L^2(\Omega;\R)}, i\in \{1,2,...,N\}$. Using basic properties of the conditional expectation, one has
\begin{equation}\label{eq:I_2 vanish[0.5AS25]}
        \E  
        \bigg(  
         \int_{t_{i-1}}^{t_i} 
             \big[ g(X_s) - g_h(  X_{ t_{i-1} } ) 
         \big] \,\mathrm{d}W_s 
         \Big|  \F_{t_{i-1}} 
         \bigg)
        =
        0.
\end{equation}
The Young inequality gives
\begin{equation}
\begin{aligned}
&
    \Bigg\|
\mathbb{E} 
\bigg( 
\int_{t_{i-1}}^{t_i} 
\big[ \nu\left(X_{s^{-}}\right) 
-
\nu \left(X_{t_{i-1}}\right) \big] \mathrm{d} N_s \mid \F_{t_{i-1}}
\bigg) 
\Bigg\|_{L^2(\Omega ; \mathbb{R})}^2 
\\
& \leq 
2\left\|\mathbb{E}
\left(
\int_{t_{i-1}}^{t_i} 
\big[ \nu\left(X_{s^{-}}\right) -\nu \left(X_{t_{i-1}}\right)  \big]
\mathrm{d} \tilde{N}_s \mid \F_{t_{i-1}}\right)\right\|_{L^2(\Omega ; \mathbb{R})}^2 
\\
& \quad
+2\left\|
\lambda 
\mathbb{E}\left(
\int_{t_{i-1}}^{t_i} 
\big[ \nu\left(X_{s^{-}}\right) -\nu \left(X_{t_{i-1}}\right) \big]
\mathrm{d} s \mid \F_{t_{i-1}}\right)\right\|_{L^2(\Omega ; \mathbb{R})}^2\\
& =
2 \lambda^2 
\left\|
\int_{t_{i-1}}^{t_i} 
\big[ \nu\left(X_{s^{-}}\right) -\nu \left(X_{t_{i-1}}\right) \big]
\mathrm{d} s 
\right\|_{L^2(\Omega ; \mathbb{R})}^2.
\end{aligned}
\end{equation}
The H\"older inequality, the Lipschitz condition of $\nu$ \eqref{eq:nu_assmption1[0.5AS25]} and Lemma \ref{lem:Holder_continuous_X[0.5AS25]} further lead to
\begin{equation}
\begin{aligned}
&
\Bigg\|
\mathbb{E}
\bigg(
\int_{t_{i-1}}^{t_i} 
\big[ \nu\left(X_{s^{-}}\right) 
-
\nu  \left(X_{t_{i-1}}\right)  \big] \mathrm{d} N_s \mid \F_{t_{i-1}}
\bigg)
\Bigg\|_{L^2(\Omega ; \mathbb{R})}^2 
\\
\leq 
& 2 \lambda^2 h 
\int_{t_{i-1}}^{t_i}\left\|
\nu\left(X_{s^{-}}\right) -\nu \left(X_{t_{i-1}}\right)
\right\|_{L^2(\Omega ; \mathbb{R})}^2 \mathrm{~d} s \\
\leq 
& C M^2 h 
\int_{t_{i-1}}^{t_i}\left\|X_{s^{-}}-X_{t_{i-1}}\right\|_{L^2(\Omega ; \mathbb{R})}^2 \mathrm{~d} s \\
\leq 
& C h^3 .
\end{aligned}
\end{equation}
Moreover, following the same line of estimation arguments as in the analysis of the term $I_2$,
one obtains for any $s \in [t_{i-1}, t_i]$ that
\begin{equation}
    \begin{split}
    &
      \| f(X_s) 
        - 
        f( X_{ t_{i-1} } ) 
     \|^2_{L^2(\Omega;\R)}  \\
&\leq
        4 
        \bigg\| 
        \int_{t_{i-1}}^{s} f'(X_l) F(X_l) \mathrm{~d}l
        \bigg\|^2_{L^2(\Omega;\R)} 
        +
        4
        \bigg\| 
        \int_{t_{i-1}}^{s} f'(X_l) g(X_l) \mathrm{~d}W_l
        \bigg\|^2_{L^2(\Omega;\R)} 
\\
&\quad
        +
        \bigg\|
        \int_{t_{i-1}}^{s}
         f''(X_l) g(X_l)^2 \mathrm{~d}l
        \bigg\|^2_{L^2(\Omega;\R)} 
        +
        4 
        \bigg\|
        \int_{t_{i-1}}^{s}
        [ f( X_{l^-} + \nu(X_{l^-})) - f(X_{l^-} ) ] \mathrm{~d} N_l
        \bigg\|^2_{L^2(\Omega;\R)} 
\\
&\leq
        4 h 
        \int_{t_{i-1}}^{s} 
        \bigg\| 
        f'(X_l) F(X_l) 
        \bigg\|^2_{L^2(\Omega;\R)} 
        \mathrm{~d}l
        +
        4
        \int_{t_{i-1}}^{s} 
        \bigg\| 
        f'(X_l) g(X_l)
        \bigg\|^2_{L^2(\Omega;\R)} 
        \mathrm{~d} l
\\
&\quad
        +
         h
        \int_{t_{i-1}}^{s}
        \bigg\|
         f''(X_l) g(X_l)^2 
        \bigg\|^2_{L^2(\Omega;\R)} 
        \mathrm{~d}l
        +
        8 \lambda
        \int_{t_{i-1}}^{s}
        \bigg\|
        f( X_{l^-} + \nu(X_{l^-})) - f(X_{l^-} ) 
        \bigg\|^2_{L^2(\Omega;\R)} 
        \mathrm{~d} l
\\
&\quad
        +
        8 \lambda^2 h 
        \int_{t_{i-1}}^{s}
        \bigg\|
        f( X_{l^-} + \nu(X_{l^-})) - f(X_{l^-} ) 
        \bigg\|^2_{L^2(\Omega;\R)} 
        \mathrm{~d} l
\\
&\leq
Ch \Big(1  + \mathbbm{1}_{ \{ r<2\} } \sup_{t \in [0,T]} \|X_{t}^{-1}\|^{4-2r}_{L^{4-2r}(\Omega;\R)}  +\sup_{t \in [0,T]} \|X_{t}\|^{4r-2}_{L^{4r-2}(\Omega;\R)} \Big).
    \end{split}
\end{equation}
Based on the preceding estimates, 
by adding and subtracting the terms
$f(X_{t_{i-1}})$
to the integrands of $R_i$
one shows that
\begin{equation}\label{eq:Ri|F_i-1<Ch^2[0.5AS25]}
\begin{split}
    \| \E (R_i| \F_{i-1}) \|_{L^2(\Omega;\R)}
    &
    \leq
    \alpha_{-1}
       \int_{t_{i-1}}^{t_i}
       \|  X_{s}^{-1} - X_{ t_{i} } ^{-1} 
       \|_{L^2(\Omega;\R)} 
       \mathrm{~d}s
    +
    \alpha_1 
    \int_{t_{i-1}}^{t_i}
    \| 
       X_s -  X_{ t_{i-1} } 
    \|_{L^2(\Omega;\R)}
    \mathrm{~d}s 
    \\
    & \quad
    +
      \int_{t_{i-1}}^{t_i}
    \| f(X_s) 
        - 
        f( X_{ t_{i-1} } ) 
     \|_{L^2(\Omega;\R)}
    \mathrm{~d}s 
    +
    \int_{t_{i-1}}^{t_i}
    \big\| f( X_{ t_{i-1} } ) - f_h(  X_{ t_{i-1} } )
    \big\|_{L^2(\Omega;\R)}
    \mathrm{~d}s
    \\
    & \quad
  +
     \left\|
\mathbb{E}
\left(
\int_{t_{i-1}}^{t_i} 
\nu\left(X_{s^{-}}\right) -\nu \left(X_{t_{i-1}}\right) \mathrm{d} N_s \mid \F_{t_{i-1}}
\right)
\right\|_{L^2(\Omega ; \mathbb{R})}
        \\
    & 
    \leq
   Ch^{\frac32} \Big(1  + \mathbbm{1}_{ \{ r< 2
   \} } \sup_{t \in [0,T]} \|X_{t}^{-1}\|^{2-r}_{L^{4-2r}(\Omega;\R)}  +\sup_{t \in [0,T]} \|X_{t}\|^{2r}_{L^{4r}(\Omega;\R)} \Big)
 ,
\end{split}
\end{equation}
where  
\ref{lem:Holder_continuous_X[0.5AS25]} and
\eqref{eq:assump.f(x) - f_h(x)[0.5AS25]} were also used in the last inequality.
Similar to \eqref{eq:R_i_L2_estimate[0.5AS25]},
employing Lemma \ref{lem:moment_case1[0.5AS25]} yields
    \begin{equation}
        \| \E [ R_{n+1} | \F_n ] \|_{L^2(\Omega;\R)}
        \leq
        C h^\frac32, 
        \quad
        \text{for all } \ n \in \{0,1,...,N-1 \} , N \in \N.
    \end{equation}
This completes the proof.
\end{proof}

Bearing Theorem \ref{thm(main):Order_onehalf[0.5AS25]} and Lemma \ref{lem. R_n estimates [0.5AS25]} in mind, 
the strong convergence result of order $1/2$ can be obtained directly.

\begin{thm}\label{thm(revised_last):Order_onehalf[0.5AS25]}
    Let $\{ X_t \}_{ t\in [0,T] }$ and $\{ Y_n \}_{n=0}^N $ be the solutions of \eqref{eq:ait-sahalia[0.5AS25]} and \eqref{eq:numeri_method[0.5AS25]}, respectively. Let Assumption 
    \ref{assump:modification_f-g[0.5AS25]} 
    hold. 
If one of the following conditions stands:
\begin{enumerate}[(1)]
    \item $r + 1 > 2 \rho$,
    \item $r + 1 = 2 \rho$, $\tfrac{\alpha_2}{\sigma^2} 
    \geq
    2r - \tfrac12$,
\end{enumerate}
    then there exists a constant $C$ independent of $h$, such that for all $n \in \{1,2,...,N \} $, $N \in \N$,
\begin{equation}
    \| X_{ t_{n} } - Y_n \|_{L^2(\Omega;\R)} 
    \leq C 
    h^\frac12.
\end{equation}

\end{thm}

\begin{rem}
    We point out that the parameter condition $\tfrac{\alpha_2}{\sigma^2} \geq 2r - \tfrac{1}{2}$ imposed in Lemma \ref{lem. R_n estimates [0.5AS25]} and Theorem \ref{thm(revised_last):Order_onehalf[0.5AS25]} is stricter than \eqref{eq:parameter_conditions[0.5AS25]},
    since
\[
\tfrac{\alpha_2}{\sigma^2} \geq 2r - \tfrac{1}{2} > \tfrac{r + 3}{8} \geq \tfrac{1}{8}(r + \tfrac{1}{r} + 2).
\]
Therefore, it suffices to check the conditions in Theorem \ref{thm(revised_last):Order_onehalf[0.5AS25]} to attain the final convergence rate.

\end{rem}

\section{Numerical experiments}\label{section:numericalexperiments[0.5AS25]}

This section is devoted to 
some numerical experiments to verify the previous theoretical results. 
Consider a generalized  A\"{i}t-sahalia-type interest rate model with Poisson jumps of the form
\begin{equation}
     \mathrm{d} X_t 
     =
     ( \alpha_{-1} X_t^{-1} - \alpha_0 + \alpha_1 X_t - \alpha_2 X_t^r ) \mathrm{~d}t 
     +
     \sigma X_t^{\rho} \mathrm{~d} W_t
     +
     \nu( X_{t^-} ) \mathrm{~d} N_t
     ,
     \ t >0, \quad  X_0 = 1.
\end{equation}
Here, we take $\nu(x)=0.5x, x \in \R$
so that \eqref{eq:nu_assmption1[0.5AS25]}-\eqref{eq:nu_assmption2[0.5AS25]} are satisfied. 
Two sets of parameters are chosen, which meet the conditions required by Theorem \ref{thm(revised_last):Order_onehalf[0.5AS25]}:

\begin{itemize}
    \item Example 1 (non-critical case $r + 1 > 2\rho$): $\alpha_{-1} = \frac{3}{2}$, $\alpha_0 = 2$, $\alpha_1 = 1$, $\alpha_2 = 3$, $\sigma = 1$, $r = 5$, $\rho = 2$;
    \item Example 2 (critical case $r = 3$, $\rho = 2$): $\alpha_{-1} = \frac{3}{2}$, $\alpha_0 = 2$, $\alpha_1 = 1$, $\alpha_2 = 3$, $\sigma = 1$, $r = 3$, $\rho = 2$.
\end{itemize}

Throughout the tests, two types of the modification functions \(f_h\) and \(g_h\), as described in Example \ref{example1} and Example \ref{example2[0.5AS25]}, are employed to produce the numerical approximations.
The corresponding schemes are listed as follows:
\begin{itemize}
    \item 
    The tamed Euler method (TEM):
    for $n \in \{0, 1, 2, ..., N - 1\}$, 
    \begin{align}
 \label{form1}
 Y_{n+1} 
=
Y_n + (\alpha_{-1} Y_{n+1}^{-1} - \alpha_0 + \alpha_1 Y_n - \frac{\alpha_2 Y_n^r h }{1+h^\frac{1}{2}Y_n^r})h 
 + \frac{\sigma Y_n^\rho}{1+h^\frac{1}{2}Y_n^r} \Delta W_n  +\nu( Y_n ) \Delta N_n,
\end{align}
    \item 
    The projected Euler method (PEM):
    set $\P_h(x) = \min\{1, h^{-\frac{1}{2r - 2}} |x|^{-1}\} x, x \in \R$. For $n \in \{0, 1, 2, ..., N - 1\}$, 
    \begin{equation}
Y_{n+1} =  Y_n + (\alpha_{-1} Y_{n+1}^{-1} - \alpha_0 + \alpha_1 Y_n - \alpha_2 (\P_h(Y_n))^r h 
 + \sigma (\P_h(Y_n))^\rho \Delta W_n  +\nu( Y_n ) \Delta N_n.
\end{equation}
\end{itemize}

In what follows,
we investigate the strong convergence rates of the proposed schemes over a fixed time interval
\( T = 1 \), with expectations approximated using $10^4$ sample paths.
The numerical approximations are generated by the two methods described above as well as the backward Euler-Maruyama method (BEM) \cite{zhao2021backward} under a sequence of time step sizes $h = 2^{-i},i = 5, 6, 7, 8, 9, 10$.
For comparison, the “exact" solutions are simulated by BEM, which is proven to be strongly convergent with order $1/2$ for the considered model, with a sufficiently small  step-size \( h_{\text{exact}} = 2^{-14} \). 
The estimation error is computed in terms of $e_h := \left( \sup_{0 \leq n \leq N} \mathbb{E} \left[ |X_{t_n} - Y_n|^2 \right] \right)^{\frac{1}{2}}$.
\\

\begin{figure}[htbp]
    \centering
    \begin{minipage}{0.49\textwidth}
        \centering
        \includegraphics[width=\textwidth]{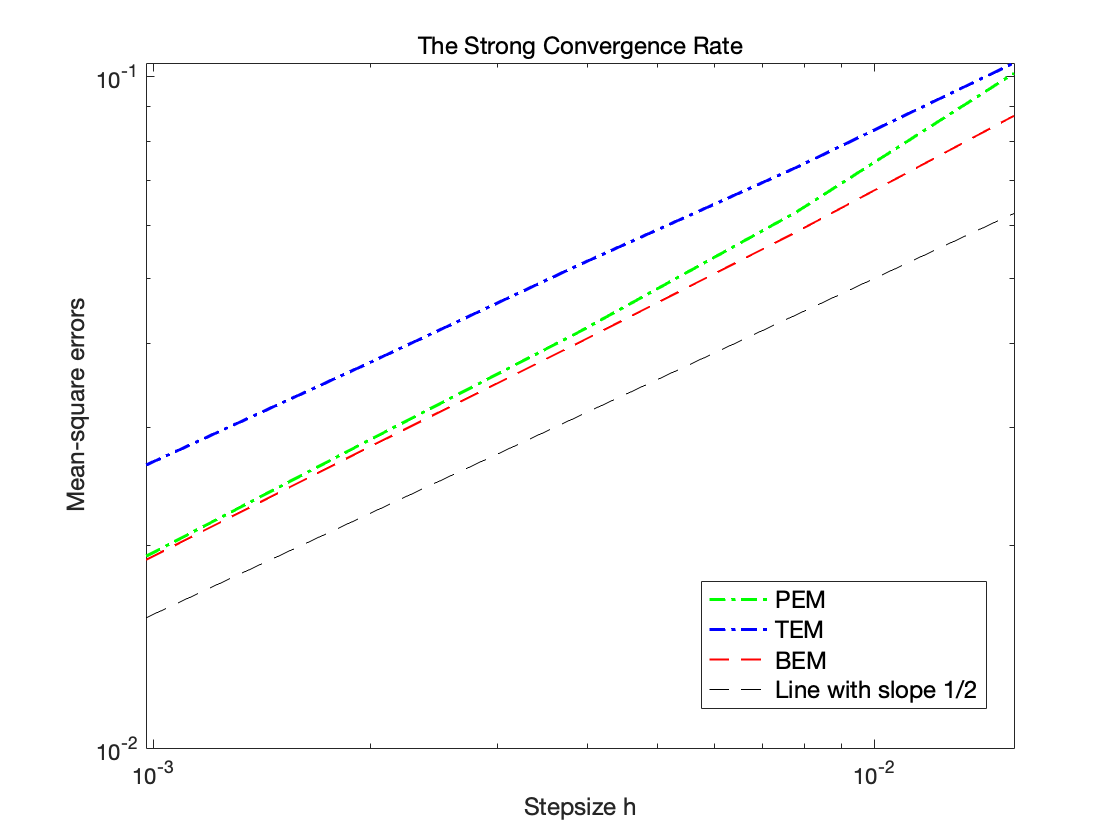} 
      	\caption*{\small{(a) Example 1: non-critical case}} 
        \label{results for case1_tamed}
    \end{minipage}
    \hfill 
    \begin{minipage}{0.49\textwidth}
        \centering
        \includegraphics[width=\textwidth]{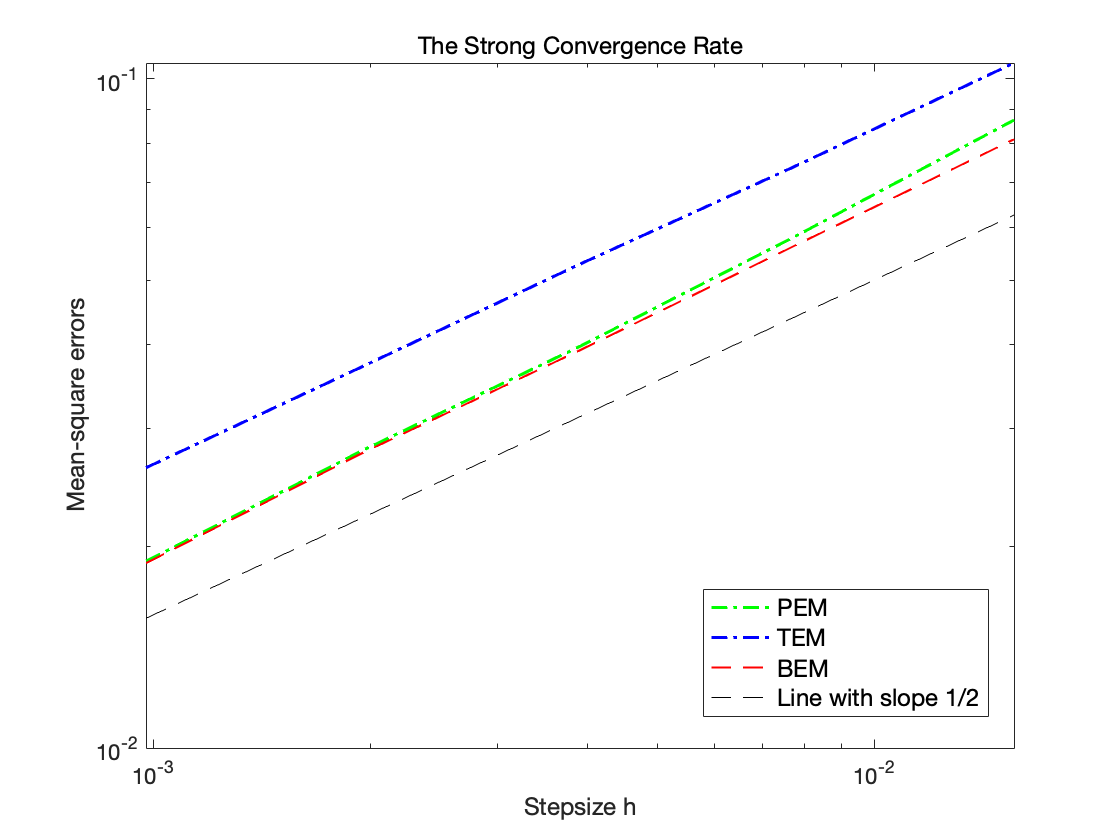} 
       \caption*{\small{(b) Example 2: critical case}}
        \label{results for case2_tamed}
    \end{minipage}
    \caption{Mean-square convergence rates}
    \label{fig:results}
\end{figure}

Fig.~\ref{fig:results} presents the mean-square convergence rates of TEM, PEM and BEM for both the non-critical case \( r + 1 > 2\rho \) and the critical case \( r + 1 = 2\rho \).
As shown, all errors are plotted on a log–log scale.
Compared with the reference lines, the numerical results demonstrate that all the three schemes achieve mean-square convergence rates close to $1/2$,
which can be detected more transparently from Table \ref{tab:least_squares_fit[0.5AS25]}.
This agrees with the theoretical results of Theorem \ref{thm(revised_last):Order_onehalf[0.5AS25]}.
It is worth noting that in the experiments of PEM, 
the projection mechanism was rarely triggered, since 
the positivity of the solution is largely maintained by the inherent structure of the scheme, rather than by frequent projections.
This reflects one of the key advantages of our scheme design.

Moreover, the time costs of TEM and BEM are reported in Table \ref{tab:time_costs[0.5AS25]}. 
The results clearly show that the proposed explicit strategy significantly reduce the computational cost compared to the implicit methods.
This advantage becomes particularly pronounced as the exponent $r$ in the drift term increases, leading to more expensive nonlinear equation solvers in the implicit scheme. 
These observations highlight the practical efficiency of our schemes.

\begin{table}[htp]
	\centering
	\setlength{\tabcolsep}{8mm}
        \caption{Time costs (seconds) for TEM and BEM over $10^4$ Brownian paths}\label{tab:time_costs[0.5AS25]}
	\begin{tabular}{c c c }
		\toprule[2pt]
		  &  TEM  & BEM \\
		\midrule 
           Example 1: & time = $18.150$  & time = $43.657$\\
		 Example 2: &  time = $17.988$ &  time = $37.003$ \\
		\bottomrule [2pt]
	\end{tabular}
	\vspace{2pt}
\end{table}

\begin{table}[htp]
	\centering
 \footnotesize
	\setlength{\tabcolsep}{7mm}
        \caption{A least square fit for the convergence rate $q$}\label{tab:least_squares_fit[0.5AS25]}
	\begin{tabular}{m{2cm}<{\centering} m{2.5cm}<{\centering} m{2.5cm}<{\centering} m{2.5cm}<{\centering} }
		\toprule[2pt]
		  & PEM & TEM & BEM \\
		\midrule 
		 Example 1 
         & $q=0.5923$, $resid=0.0579$ 
         & $q=0.4959$, $resid=0.0100$ 
         & $q=0.5473$, $resid=0.0156$
         \\
		 Example 2 
         & $q=0.5445$, $resid=0.0205$ 
         & $q=0.5014$, $resid=0.0059$ 
         & $q=0.5239$, $resid=0.0123$ 
         \\
		\bottomrule [2pt]
	\end{tabular}
	\vspace{2pt}
 \label{table1}
\end{table}

\section{Conclusion}

In this paper, we propose a new class of explicit schemes for approximating the generalized A\"it-Sahalia type model with Poisson jumps. The proposed schemes are shown to be unconditionally positivity-preserving (namely, no further restriction is imposed on the time step size $h>0$), and admitting a strong convergence rate of order $1/2$ for both the non-critical case $\gamma + 1 > 2 \rho$ and the critical case $\gamma + 1 = 2 \rho$. Also, we note that the analysis does not rely on any a priori moment bounds of the schemes. 
Finally,
numerical experiments are conducted to verify the theoretical results, where two concrete numerical methods are applied, as the examples of our proposed schemes.

\section*{Acknowledgements}
The authors would like to thank Prof. Xiaojie Wang for his insightful comments that helped improve this manuscript.

\vskip6mm
\bibliography{ref}

\end{document}